\documentclass[oneside,11pt]{amsart}
\usepackage[a4paper]{geometry}
\usepackage{amsthm}
\theoremstyle{plain} 

\usepackage{ragged2e}

\newtheorem{theorem}{Theorem}[section]
\newtheorem{lemma}{Lemma}[section] 
\newtheorem{corollary}{Corollary}[section] 
\newtheorem{remark}{Remark}[section]

\usepackage{amsmath}
\usepackage{amsfonts}
\usepackage{graphics}
\usepackage{amssymb}
\usepackage{verbatim}
\usepackage{amscd}
\usepackage{enumerate}
\usepackage{graphicx}

\usepackage{hyperref}

\usepackage{bbm}

\usepackage{enumitem}

\DeclareMathOperator\arctanh{arctanh}

\begin{document}

\title {A remark on Carleson measures of domains in $\mathbb{C}^{n}$}


  
\author[Phung Trong Thuc]{Phung Trong Thuc}

\thanks{\fontsize{7.3}{7.3}\selectfont FACULTY OF APPLIED SCIENCE--HO CHI MINH CITY UNIVERSITY OF TECHNOLOGY, VIETNAM\newline \hspace*{0.32cm} \fontsize{9}{9}\selectfont E-mail address: ptrongthuc@hcmut.edu.vn}

 

\begin{abstract} We  provide characterizations of Carleson measures on a certain class of bounded pseudoconvex domains. An example of a vanishing Carleson measure whose Berezin transform does not vanish on the boundary is given in the class of the Hartogs triangles 
$$\mathbb{H}_{k}:=\left\{ \left(z_{1},z_{2}\right)\in\mathbb{C}^{2}:\left|z_{1}\right|^{k}<\left|z_{2}\right|<1\right\},\;k\in\mathbb{Z}^{+}.$$

\end{abstract}

\maketitle

\section{Introduction}

Let $\Omega\,\subset\,\mathbb{C}^{n}$ be a bounded domain, and let $A^{2}\left(\Omega\right)$ be the space of holomorphic, square-integrable functions on $\Omega$ (with respect
to the Lebesgue measure $dV$ in $\mathbb{C}^{n}$). A non-negative finite  Borel measure $\mu$ on $\Omega$ is called \emph{a  Carleson measure} of $A^{2}\left(\Omega\right)$ if there exists a positive constant $C$ such that 
\begin{equation}\label{intro1}
\intop_{\Omega}\left|h\right|^{2}d\mu\leq C\intop_{\Omega}\left|h\right|^{2}dV,\;\forall h\in A^{2}\left(\Omega\right).
\end{equation}
Equivalently, the inclusion $A^{2}\left(\Omega\right)\hookrightarrow L^{2}\left(\Omega,\mu\right)$ 
 is bounded. A Carleson measure $\mu$ on $\Omega$ is called \emph{vanishing} if the inclusion $A^{2}\left(\Omega\right)\hookrightarrow L^{2}\left(\Omega,\mu\right)$  is compact.  
  
We are interested in finding a characterization of Carleson and of vanishing Carleson measures on a bounded pseudoconvex domain $\Omega$ in $\mathbb{C}^{n}$. Let $K$ denote the Bergman kernel of $\Omega$. Given a Borel measure
$\mu$ on $\Omega$, the Berezin transform of $\mu$ is defined by
\begin{align*}
\mathcal{B}_{\mu}:\Omega & \rightarrow\mathbb{R}^{+}\cup\left\{ \infty\right\} \\
w & \rightarrow\intop_{\Omega}\frac{\left|K\left(z,w\right)\right|^{2}}{K\left(w,w\right)}d\mu\left(z\right).
\end{align*}
Note that since $\Omega$ is bounded, this function is well-defined. Let $k_{w}$ be the normalized Bergman kernel, defined by
\[
k_{w}\left(z\right):=\frac{K\left(z,w\right)}{\sqrt{K\left(w,w\right)}}.
\]
Since $k_{w}\in A^{2}\left(\Omega\right)$ and \begin{equation}
\intop_{\Omega}\left|k_{w}\left(z\right)\right|^{2}dV\left(z\right)=1,\;\forall w\in\Omega,\label{eq:keye1}
\end{equation}
a necessary condition for \eqref{intro1} to be satisfied is that
$\mathcal{B}_{\mu}\in L^{\infty}\left(\Omega\right)$. This leads
us to the question whether this necessary condition is also sufficient.
 
Carleson measures have been studied for certain classes of pseudoconvex
domains, particularly in connection with many important operators in
complex analysis such as: Hankel  and  Toeplitz operators, see e.g. \cite{BBCZ90,Li92,Li94,ARS12}; and its
generalisation to other spaces in harmonic analysis \cite{CiWo82,ARS02,Tch08,ARS08}.  For bounded symmetric domains (e.g. unit polydiscs), B{\'e}koll{\'e}-Berger-Coburn-Zhu
\cite[Theorem 8]{BBCZ90} proved that $\mathcal{B}_{\mu}\in L^{\infty}$
is also a sufficient condition for a Carleson measure $\mu$.  For strongly
pseudoconvex domains with smooth boundary, the same property was confirmed
by H. Li  \cite[Theorem C]{Li92}. In this case of domains, Abate
and Saracco \cite[Theorem 1.1]{AS11} proved that the boundedness
of $\mathcal{B}_{\mu}$ on $\Omega$ is also a necessary and sufficient
condition for Carleson measures of $A^{p}\left(\Omega\right)$, for
any $p\geq1$.  Turning to the question of vanishing Carleson measures, it has been shown that (see \cite{BBCZ90,Li92,ARS12}), for any strongly pseudoconvex domain or  bounded symmetric domain $\Omega$, $\mu$ is a vanishing Carleson measure if and only if ${\displaystyle \lim_{z\rightarrow\partial\Omega}}\mathcal{B}_{\mu}\left(z\right)=0$. Here $\partial\Omega$ denotes the topological boundary of $\Omega$. The ``only if'' statement comes easily from the fact that $k_{z}\rightarrow0$
weakly in $A^{2}\left(\Omega\right)$ as $z\rightarrow\partial\Omega$.

To our knowledge; however, it is not yet much studied to what extent these characterizations are still true for a general pseudoconvex domain in $\mathbb{C}^{n}$. For the sake of illustration, let us first consider the class of bounded convex domains in $\mathbb{C}^{n}$ (with no boundary regularity assumptions). In general, a domain in this class is neither
a strongly pseudoconvex domain nor a symmetric domain (for example, consider the Thullen domains $\left\{ \left(z_{1},z_{2}\right)\in\mathbb{C}^{2}:\left|z_{1}\right|^{p}+\left|z_{2}\right|^{q}<1\right\} $,
 $p,q\geq1$). Thus in this case, it is reasonable
to ask whether the same behaviours of the Berezin transform would capture
the characterisations of Carleson measures. We shall show that this is the case. The method we give here is elementary.
The argument combines known results for intrinsic geometry on sublevel
sets of the pluricomplex Green function and an estimate due to B{\l}ocki \cite{Blok14}. Our approach also extends previous work on strongly pseudoconvex or bounded symmetric domains to a wider class of domains.

Let $G_{\Omega}\left(\cdot,w\right)$ be the pluricomplex Green function
with pole $w\in\Omega$, defined by
\[
G_{\Omega}\left(\cdot,w\right):=\sup\left\{ u\left(\cdot\right):u\in PSH^{-}\left(\Omega\right),\limsup_{z\rightarrow w}\left(u\left(z\right)-\log\left|z-w\right|\right)<\infty\right\} .
\]
Here $PSH^{-}\left(\Omega\right)$ denotes the set of all negative
plurisubharmonic functions on $\Omega$. The following useful estimate was proved by B{\l}ocki \cite{Blok14}, which is  a sharper version of a previous estimate of Herbort
\cite{Her99}: 
for any bounded pseudoconvex domain $\Omega\,\subset\,\mathbb{C}^{n}$  and $M>0$,
\begin{equation}
\intop_{\left\{ G_{\Omega}\left(\cdot,z\right)<-M\right\} }\left|h\left(w\right)\right|^{2}dV\left(w\right)\geq e^{-2nM}\frac{\left|h\left(z\right)\right|^{2}}{K\left(z,z\right)},\;\forall z\in\Omega,\,\forall h\in\mathcal{O}\left(\Omega\right).\label{eq:He-Bo}
\end{equation}
Using this estimate, it follows that 
\begin{equation}
\intop_{\Omega}\left|h\left(z\right)\right|^{2}d\mu\left(z\right)\leq e^{2nM}\intop_{\Omega}\intop_{\Omega}\mathbbm{1}_{A_{z,M}}\left(w\right)K\left(z,z\right)\left|h\left(w\right)\right|^{2}dV\left(w\right)d\mu\left(z\right),\label{eq:kein2}
\end{equation}
where $A_{z,M}:=\left\{ G_{\Omega}\left(\cdot,z\right)<-M\right\} $.  Note that it is safe to write the iterated integral on the RHS of
\eqref{eq:kein2}, because $G_{\Omega}$ is upper semicontinuous on
$\Omega\times\Omega$ (see \cite[Corollary 6.2.6]{Kli91}), so the
integrand is measurable.

From this we immediately obtain:

\begin{theorem}\label{thm:re1}
Let $\Omega$ be a bounded pseudoconvex domain in $\mathbb{C}^{n}$.
Assume that there are $M>0$ and $c<1$ such that $d_{S}\left(z,w\right)<c$,
for any $w\in A_{z,M}$ and any $z\in\Omega$; where $d_{S}$ is the
Skwarczy{\'n}ski distance defined by
\[
d_{S}\left(z,w\right):=\left(1-\frac{\left|K\left(z,w\right)\right|}{\sqrt{K\left(z,z\right)}\sqrt{K\left(w,w\right)}}\right)^{\frac{1}{2}};\;\;z,w\in\Omega.
\]
Then 
\[
\intop_{\Omega}\left|h\left(z\right)\right|^{2}d\mu\left(z\right)\leq C\intop_{\Omega}\mathcal{B}_{\mu}\left(z\right)\left|h\left(z\right)\right|^{2}dV\left(z\right),\;\forall h\in A^{2}\left(\Omega\right).
\]
Here $C:=\frac{e^{2nM}}{\left(1-c^{2}\right)^{2}}$. As a consequence, $\mu$ is a Carleson measure iff $\mathcal{B}_{\mu}\in L^{\infty}\left(\Omega\right)$.

\end{theorem}

\begin{proof}
For $w\in A_{z,M}$, we have
\[
\frac{\left|K\left(z,w\right)\right|}{\sqrt{K\left(z,z\right)}\sqrt{K\left(w,w\right)}}>1-c^{2}.
\]
Using \eqref{eq:kein2}, we see that
\begin{align*}
\intop_{\Omega}\left|h\left(z\right)\right|^{2}d\mu\left(z\right) & \leq\frac{e^{2nM}}{\left(1-c^{2}\right)^{2}}\intop_{\Omega}\intop_{\Omega}\frac{\left|K\left(z,w\right)\right|^{2}}{K\left(w,w\right)}\left|h\left(w\right)\right|^{2}d\mu\left(z\right)dV\left(w\right)\\
 & = C\intop_{\Omega}\mathcal{B}_{\mu}\left(w\right)\left|h\left(w\right)\right|^{2}dV\left(w\right),
\end{align*}
as desired.

\end{proof}

Let us apply this to the case of bounded convex domains $\Omega$ in $\mathbb{C}^{n}$. It is known that (see \cite[Corollary 6.5.3]{Kli91}) \[
A_{z,M}\subset\left\{ w\in\Omega:d_{C}\left(w,z\right)<\arctanh\left(e^{-M}\right)\right\} ,
\]
where $d_{C}$ is the Carath{\'e}odory distance. Since $\Omega$ is convex, by a well-known result of Lempert \cite{Lem81}, we have $d_{C}\equiv d_{K}$, where $d_{K}$ is the Kobayashi distance. It follows that
\[
A_{z,M}\subset\left\{ w\in\Omega:d_{K}\left(w,z\right)<\arctanh\left(e^{-M}\right)\right\} .
\]
On the other hand, $d_{S}\leq\left.d_{B}\right/\sqrt{2}$,
see \cite[Corollary 6.4.7]{JP93}, where $d_{B}$ denotes the Bergman
distance. Note that this fact is true for a general bounded pseudoconvex domain. Finally, it is also known that the Bergman and the Kobayashi  metrics are equivalent on bounded convex domains, e.g. \cite{NP03}. Thus $d_{B}\leq Cd_{K}$,
for some positive constant $C$. Therefore, for $M$ large enough, $d_{S}\left(w,z\right)<\left.1\right/2$,
for any $w\in A_{z,M}$, so the condition of Theorem \ref{thm:re1} is satisfied. Thus we obtain the following corollary:

\begin{corollary}
A measure $\mu$ on a bounded convex domain $\Omega$ is a Carleson measure  iff
$\mathcal{B}_{\mu}\in L^{\infty}\left(\Omega\right)$.
\end{corollary}

Using the estimate \eqref{eq:He-Bo}, we now claim the characterization
of vanishing Carleson measures on convex domains:

\begin{corollary}\label{corconv}
A  Carleson measure $\mu$ on a bounded convex domain $\Omega$ is  vanishing  
iff $\mathcal{B}_{\mu}\left(z\right)\rightarrow0$ as $z\rightarrow\partial\Omega$.

\end{corollary}

\begin{proof}
Since $k_{w}\rightarrow0$ weakly in $A^{2}\left(\Omega\right)$
as $w\rightarrow\partial\Omega$ (see \cite[Lemma 4.9]{CS18}),
the ``only if'' part follows.

In the converse direction, notice that $\mathcal{B}_{\mu}$
is continuous on $\Omega$. This follows from the dominated
convergence theorem. The hypothesis thus
gives $\mathcal{B}_{\mu}\in C\left(\overline{\Omega}\right)$.
Choose a sequence of domains $\left\{ \Omega_{j}\right\} $ such that
$\overline{\Omega}_{j}\Subset\Omega_{j+1}\Subset\Omega$ and
$\cup_{j=1}^{\infty}\Omega_{j}=\Omega$. For each $j$, the
function $R_{j}$, defined by
\[
R_{j}\left(w\right):=\intop_{\Omega}\frac{\left|K\left(z,w\right)\right|^{2}}{K\left(w,w\right)}\left(1-\mathbbm{1}_{\Omega_{j}}\left(z\right)\right)\,d\mu\left(z\right),\;w\in\Omega,
\]
is also continuous on $\overline{\Omega}$ for the same reason.
Again, by the dominated convergence theorem, $R_{j}\searrow0$ pointwise.
 From Dini's theorem, we conclude that $R_{j}\rightarrow0$ uniformly
on $\Omega$ as $j\rightarrow\infty$. On the other hand, for $h\in A^{2}\left(\Omega\right)$,
we have 
\begin{align*}
\intop_{\left.\Omega\right\backslash \Omega_{j}}\left|h\right|^{2}d\mu & \leq e^{2nM}\intop_{\Omega}\intop_{\Omega}\left(1-\mathbbm{1}_{\Omega_{j}}\left(z\right)\right)\mathbbm{1}_{A_{z,M}}\left(w\right)K\left(z,z\right)\left|h\left(w\right)\right|^{2}d\mu\left(z\right)dV\left(w\right)\\
 & \leq C\intop_{\Omega}\intop_{\Omega}\frac{\left|K\left(z,w\right)\right|^{2}}{K\left(w,w\right)}\left(1-\mathbbm{1}_{\Omega_{j}}\left(z\right)\right)\left|h\left(w\right)\right|^{2}d\mu\left(z\right)dV\left(w\right)\\
 & = C\intop_{\Omega}R_{j}\left(w\right)\left|h\left(w\right)\right|^{2}dV\left(w\right).
\end{align*}
Therefore the operators
\begin{align*}i_{\Omega_{j}}:A^{2}\left(\Omega\right) & \rightarrow L^{2}\left(\Omega,\mu\right)\\
h & \rightarrow\mathbbm{1}_{\Omega_{j}}h
\end{align*}
converge in norm to $i_{\Omega}$ as $j\rightarrow\infty$. The desired claim now follows
since (by Montel's theorem) $i_{\Omega_{j}}$ are compact
operators.

\end{proof}

 Let us discuss an extension of this approach to other classes of bounded pseudoconvex
domains. We will see in later that the condition in Theorem \ref{thm:re1}
is true for the classical Hartogs triangle $\mathbb{H}=\left\{ \left(z_{1},z_{2}\right)\in\mathbb{C}^{2}:\left|z_{1}\right|<\left|z_{2}\right|<1\right\} $,
which is an example of a non-hyperconvex domain. Note also that this
condition is also satisfied for strongly
pseudoconvex domains. This comes from the fact that for this class
of domains, the Carath{\'e}odory metric,  the
Kobayashi metric and the Bergman metric are equivalent, see \cite{DieK70,Gra75}. Also, in view
of the above example, the condition in Theorem \ref{thm:re1} is satisfied
if the Bergman distance $d_{B}\left(z,w\right)$ can be made small
for any $w$ in sublevel sets $A_{z,M}$, for some fixed (large) constant $M$. Such an estimate has been recently obtained by Zimmer \cite[Theorem 1.10]{Zim20} for domains having \emph{bounded intrinsic geometry}. This class includes in particular homogeneous domains (so, in particular, bounded symmetric domains), finite type domains in $\mathbb{C}^{2}$, strongly
pseudoconvex domains, convex domains,   $\mathbb{C}$-convex domains which are Kobayashi hyperbolic, simply connected domains which have a complete K{\"a}hler metric with pinched negative sectional curvature, and any domain that is biholomorphic to one of the previously mentioned domains. Let $\Omega$ be a bounded domain having bounded intrinsic geometry,
then $\Omega$ must be pseudoconvex \cite[Corollary 1.3]{Zim20}.
From \cite[Theorem 1.10]{Zim20}, there are constants $C,\tau>0$
such that for any $M>0$, we have $d_{B}\left(w,z\right)<e^{C-M},$
for all $w\in A_{z,M}$ with $d_{B}\left(w,z\right)<\tau$. On the
other hand, it can be seen from the proof of Theorem 6.4 on page 19
that the condition $d_{B}\left(w,z\right)<\tau$ is automatically
satisfied for any $w\in A_{z,M}$, provided that $M$ is large enough.
We conclude that any bounded domain with bounded intrinsic geometry
enjoys the property stated in Theorem \ref{thm:re1}. Therefore $\mathcal{B}_{\mu}\in L^{\infty}\left(\Omega\right)$
is a necessary and sufficient condition for a Carleson measure $\mu$
on any  domain $\Omega$ in this class.

Now let us generalize Corollary \ref{corconv} to other classes of
domains. The following statement is clear from the proof of Corollary
\ref{corconv}.

\begin{theorem}\label{thm:re2}
Let $\Omega\subset\mathbb{C}^{n}$ be a bounded pseudoconvex domain
satisfying the condition in Theorem \ref{thm:re1}. Then the following statements hold:
\begin{enumerate}[label=(\roman*), font=\normalfont]
\item[(i)]  If a Carleson measure $\mu$ on $\Omega$ satisfies the condition: $\mathcal{B}_{\mu}\left(z\right)\rightarrow0$
as $z\rightarrow\partial\Omega$ then $\mu$ is vanishing.
\item[(ii)] If  $\mu$ is a vanishing Carleson measure on $\Omega$, and assume 
further that 
\begin{equation}
k_{z}\rightarrow0\;\text{weakly in \ensuremath{A^{2}\left(\Omega\right)}\text{ as }}\ensuremath{z\rightarrow\partial\Omega,}\label{eq:condcom}
\end{equation}
then $\mathcal{B}_{\mu}\left(z\right)\rightarrow0$
as $z\rightarrow\partial\Omega$.
\end{enumerate}
\end{theorem}

The condition \eqref{eq:condcom} is true for pseudoconvex domains with smooth boundary, and convex
domains \cite{CS18}. It is also satisfied for domains having a
certain upper bound estimate on the Bergman kernel, such as bounded symmetric
domains \cite{BBCZ90}. Consequently, the characterization $\mathcal{B}_{\mu}\left(z\right)\rightarrow0$ as $z\rightarrow\partial\Omega$
applies to these classes of domains.

However, unlike \eqref{eq:keye1}, which is true for any pseudoconvex
domain and used to obtain the necessary condition for a Carleson
measure, the condition \eqref{eq:condcom} is not the case for any pseudoconvex
domain. An example was already given in \cite{CS20}. This motivates the need to examine characterizations of Carleson measures on pseudoconvex domains with an irregular boundary.

In the rest of this note, we study the analogous questions for the fat Hartogs triangles: $$\mathbb{H}_{k}:=\left\{ \left(z_{1},z_{2}\right)\in\mathbb{C}^{2}:\left|z_{1}\right|^{k}<\left|z_{2}\right|<1\right\} ,\; k\in\mathbb{Z}^{+}.$$ 
It can be seen that the condition  \eqref{eq:condcom} is not satisfied for $\mathbb{H}_{k}$.
To check this, for simplicity of exposition, let us restrict to the
case $k=1$. Since $g\left(z\right):=\left.1\right/z_{2}\in A^{2}\left(\mathbb{H}\right)$, we have
\[
\left|\intop_{\mathbb{H}}\frac{K\left(z,w\right)}{\sqrt{K\left(w,w\right)}}\overline{g\left(z\right)}dV\left(z\right)\right|=\frac{1}{\left|w_{2}\right|\sqrt{K\left(w,w\right)}},\;\forall w\in\mathbb{H}.
\]
By using the explicit formula \eqref{eq:1}, the desired claim now follows since 
\[
\frac{1}{\left|w_{2}\right|\sqrt{K\left(w,w\right)}}\not\rightarrow 0\;\text{as }w\rightarrow\partial\mathbb{H},
\]
for example, by considering $w_{j}=\left(\frac{1}{j^{2}},\frac{1}{j}\right)\rightarrow\left(0,0\right)\in\partial\mathbb{H}$
as $j\rightarrow\infty$.

We will show, in particular, that a measure $\mu$ is a vanishing Carleson measure on $\mathbb{H}_{k}$ if and
only if  there exists $\delta>0$ such that $\left|z_{2}\right|^{\delta}\mathcal{B}_{\mu}\left(z\right)\rightarrow0$
as $z\rightarrow\partial\mathbb{H}_{k}$.  And we can construct a vanishing measure $\mu$ on $\mathbb{H}_{k}$ such that $\mathcal{B}_{\mu}\left(z\right)\not\rightarrow0$
as $z\rightarrow\partial\mathbb{H}_{k}$ (for example, on $\mathbb{H}$,
consider $d\mu\left(z\right)=\left(1-\left|z_{2}\right|^{2}\right)^{3}\left(1-\frac{\left|z_{1}\right|^{2}}{\left|z_{2}\right|^{2}}\right)^{3}dV\left(z\right)$,
see Remark \ref{remar}). The appearance of $\left|z_{2}\right|^{\delta}$  can be explained as a weighted distance to the singular point $\left(0,0\right)$. It therefore illustrates a different property compared to the previously known examples, and indicates that in general, characterizations of vanishing Carleson measures rely heavily on boundary regularity data of the domain.

Let us first recall some known facts about the Bergman kernel of $\mathbb{H}_{k}$. Using Bell's transformation rule under proper holomorphic maps, L.Edholm \cite{Edh16} established the following  formula for the Bergman kernel of $\mathbb{H}_{k}$:
 \begin{equation}
K\left(z,w\right)=\frac{p_{k}\left(s\right)t^{2}+q_{k}\left(s\right)t+s^{k}p_{k}\left(s\right)}{k\,\pi^{2}\left(1-t\right)^{2}\left(t-s^{k}\right)^{2}},\label{eq:1}
\end{equation}
for $z=\left(z_{1},z_{2}\right)$, $w=\left(w_{1},w_{2}\right)$ in
$\mathbb{H}_{k}$, where $s:=z_{1}\overline{w}_{1},$ $t:=z_{2}\overline{w}_{2}$,
\[
p_{k}\left(s\right):=\sum_{j=1}^{k-1}j\left(k-j\right)s^{j-1},\quad\text{and}\quad q_{k}\left(s\right):=\sum_{j=1}^{k}\left(j^{2}+\left(k-j\right)^{2}s^{k}\right)s^{j-1}.
\]
If $k=1$ then we set $p_{k}\equiv 0$ in the formula \eqref{eq:1}.

Let 
\[
\mathcal{P}\left(z,w\right):=\frac{t}{\left(1-t\right)^{2}\left(t-s^{k}\right)^{2}},
\]
and consider the function $T_{\mu}:\mathbb{H}_{k}\rightarrow\mathbb{R}^{+}\cup\left\{ \infty\right\} $
defined by 
\[
T_{\mu}\left(w\right):=\intop_{\mathbb{H}_{k}}\frac{\left|\mathcal{P}\left(z,w\right)\right|^{2}}{K\left(w,w\right)}d\mu\left(z\right).
\]

Our results can be stated as follows:

\begin{theorem}\label{thm:1}
Let $\mu$ be a  Borel measure on $\mathbb{H}_{k}$.
Then the following statements are equivalent:
\begin{enumerate}[label=(\roman*), font=\normalfont]
\item[(i)] $\mu$ is a Carleson measure on $\mathbb{H}_{k}$.
\item[(ii)]  $\mathcal{B}_{\mu}\in L^{\infty}\left(\mathbb{H}_{k}\right)$.
\item[(iii)]  $T_{\mu}\in L^{\infty}\left(\mathbb{H}_{k}\right)$.

\end{enumerate}

\end{theorem}

\begin{theorem}\label{thm:2}
Let $\mu$ be a Carleson measure on $\mathbb{H}_{k}$.
Then the following statements are equivalent:
\begin{enumerate}[label=(\roman*), font=\normalfont]
\item[(i)] $\mu$ is  vanishing on $\mathbb{H}_{k}$.
\item[(ii)] There exists
$\delta>0$ such that $\left|w_{2}\right|^{\delta}\mathcal{B}_{\mu}\left(w\right)\rightarrow0$
as $w\rightarrow\partial\mathbb{H}_{k}$.
\item[(iii)] There exists
$\delta>0$ such that $\left|w_{2}\right|^{\delta}T_{\mu}\left(w\right)\rightarrow0$
as $w\rightarrow\partial\mathbb{H}_{k}$.

\end{enumerate}

\end{theorem}

\begin{remark}\normalfont
The advantage of  introducing the function $T_{\mu}$ is that it is easier to work with than $\mathcal{B}_{\mu}$ because the kernel function $\mathcal{P}$ is simpler. For example, when $k=2$, we have
\[
\mathcal{P}\left(z,w\right)=\frac{z_{2}\overline{w}_{2}}{\left(1-z_{2}\overline{w}_{2}\right)^{2}\left(z_{2}\overline{w}_{2}-z_{1}^{2}\overline{w}_{1}^{2}\right)^{2}},
\]
whereas
\[
K\left(z,w\right)=\frac{\left(1+4z_{1}\overline{w}_{1}+z_{1}^{2}\overline{w}_{1}^{2}\right)z_{2}\overline{w}_{2}+z_{1}^{2}\overline{w}_{1}^{2}+z_{2}^{2}\overline{w}_{2}^{2}}{2\pi^{2}\left(1-z_{2}\overline{w}_{2}\right)^{2}\left(z_{2}\overline{w}_{2}-z_{1}^{2}\overline{w}_{1}^{2}\right)^{2}}.
\]
Note also that $\left|K\left(z,w\right)\right|\lesssim\left|\mathcal{P}\left(z,w\right)\right|$
for any $k\in\mathbb{Z}^{+}$. When $k=1$, it is clear that $\left|K\left(z,w\right)\right|\approx\left|\mathcal{P}\left(z,w\right)\right|$.
However, for $k\geq2$, $\left|K\left(z,w\right)\right|\not\approx\left|\mathcal{P}\left(z,w\right)\right|$
since $\left|\mathcal{P}\left(z,w\right)\right|\geq 1$, while $\left|K\left(z,w\right)\right|$
  may vanish inside $\mathbb{H}_{k}\times\mathbb{H}_{k}$, see \cite{EdMcN17}. We will later use $T_{\mu}$ to construct some examples.

We shall prove Theorem \ref{thm:1} and Theorem \ref{thm:2} in much
the same above arguments. We use an idea from \cite{KLT19} to obtain
estimates on sublevel sets $A_{z,M}$.  The main new ingredient now is the analysis at the singular point $\left(0,0\right)$ in the proof of Theorem \ref{thm:2}. Throughout the proofs we use the notation $F\lesssim G$ to indicate that $F\leq c\,G$, for some positive constant $c$; and the notation $F\approx G$ for the fact $c_{1}G\leq F\leq c_{2}G$, for some positive constants $c_{1},c_{2}$.  

\end{remark}

\section{Proof of Theorem \ref{thm:1}} \label{prfA}

\begin{enumerate} [leftmargin=0pt, align=left, labelwidth=!, labelsep=0pt]
\item[$\quad\left(\text{i}\right)\Rightarrow\left(\text{iii}\right)$. ]  Since $\mathcal{P}\left(\cdot,w\right)$ is holomorphic, it suffices to
verify that
\begin{equation}\label{eq:a1}
\intop_{\mathbb{H}_{k}}\frac{\left|\mathcal{P}\left(z,w\right)\right|^{2}}{K\left(w,w\right)}dV\left(z\right)\leq C,\;\forall w\in\mathbb{H}_{k},
\end{equation}
for some positive constant $C$. We have that
\begin{align*}
\intop_{\mathbb{H}_{k}}\left|\mathcal{P}\left(z,w\right)\right|^{2}dV\left(z\right) & =\intop_{\mathbb{H}_{k}}\frac{\left|z_{2}w_{2}\right|^{2}}{\left|1-z_{2}\overline{w}_{2}\right|^{4}\left|z_{2}\overline{w}_{2}-\left(z_{1}\overline{w}_{1}\right)^{k}\right|^{4}}dV\left(z\right)\\
 & =\int_{\left.\mathbb{D}\right\backslash \left\{ 0\right\} }\frac{dV\left(z_{2}\right)}{\left|z_{2}w_{2}\right|^{2}\left|1-z_{2}\overline{w}_{2}\right|^{4}}\int_{\left|z_{1}\right|<\left|z_{2}\right|^{\frac{1}{k}}}\frac{dV\left(z_{1}\right)}{\left|1-\frac{z_{1}^{k}\overline{w}_{1}^{k}}{z_{2}\overline{w}_{2}}\right|^{4}}.
\end{align*}
Here $\mathbb{D}$ denotes the unit disk in $\mathbb{C}$. Consider
the change $\xi:=\left.z_{1}^{k}\right/z_{2}$, we get that
\begin{align}
\intop_{\mathbb{H}_{k}}\left|\mathcal{P}\left(z,w\right)\right|^{2}dV\left(z\right) & =\frac{\left|w_{2}\right|^{-2}}{k}\left(\int_{\left.\mathbb{D}\right\backslash \left\{ 0\right\} }\frac{\left|z_{2}\right|^{\frac{2}{k}-2}}{\left|1-z_{2}\overline{w}_{2}\right|^{4}}dV\left(z_{2}\right)\right)\left(\int_{\mathbb{D}}\frac{\left|\xi\right|^{\frac{2}{k}-2}}{\left|1-\xi\frac{\overline{w}_{1}^{k}}{\overline{w}_{2}}\right|^{4}}dV\left(\xi\right)\right)\nonumber \\
 & \leq\frac{\left|w_{2}\right|^{-2}}{k}J\left(\overline{w}_{2}\right)J\left(\frac{\overline{w}_{1}^{k}}{\overline{w}_{2}}\right),\label{eq:ber}
\end{align}
where
\[
J\left(a\right):=\int_{\mathbb{D}}\frac{\left|\xi\right|^{\frac{2}{k}-2}}{\left|1-\xi a\right|^{4}}dV\left(\xi\right).
\]
$J\left(a\right)$ can be estimated as
\begin{align*}
J\left(a\right) & \leq\int_{\left|\xi\right|>\frac{1}{2}}\frac{\left|\xi\right|^{\frac{2}{k}-2}}{\left|1-\xi a\right|^{4}}dV\left(\xi\right)+\int_{\left|\xi\right|\leq\frac{1}{2}}\frac{\left|\xi\right|^{\frac{2}{k}-2}}{\left|1-\xi a\right|^{4}}dV\left(\xi\right)\\
 & \lesssim\int_{\mathbb{D}}\frac{1}{\left|1-\xi a\right|^{4}}dV\left(\xi\right)+\int_{\left|\xi\right|\leq\frac{1}{2}}\left|\xi\right|^{\frac{2}{k}-2}dV\left(\xi\right)\\
 & \lesssim\left(1-\left|a\right|^{2}\right)^{-2}+1\\
 & \lesssim\left(1-\left|a\right|^{2}\right)^{-2}.
\end{align*}
Here the third inequality follows from Forelli-Rudin estimates,
see e.g. \cite[Theorem 1.3]{LiuC15}. Applying this to \eqref{eq:ber}, we obtain
\[
\intop_{\mathbb{H}_{k}}\left|\mathcal{P}\left(z,w\right)\right|^{2}dV\left(z\right)\lesssim\left|w_{2}\right|^{-2}\left(1-\left|w_{2}\right|^{2}\right)^{-2}\left(1-\frac{\left|w_{1}\right|^{2k}}{\left|w_{2}\right|^{2}}\right)^{-2}\approx K\left(w,w\right).
\]
Thus the estimate \eqref{eq:a1} holds.
\vspace*{0.3cm}
\item[$\quad\left(\text{iii}\right)\Rightarrow\left(\text{i}\right)$. ] We first observe the following elementary fact:

\emph{Fact}. For any $a,b\in\mathbb{D}$ such that 
\[
\left|\frac{a-b}{1-a\overline{b}}\right|<\frac{1}{e}
\]
 then 
\begin{equation}
\left|1-a\overline{b}\right|\approx1-\left|b\right|^{2},\label{eq:sf1}
\end{equation}
and 
\begin{equation}
1-\left|a\right|^{2}\approx1-\left|b\right|^{2}.\label{eq:sf2}
\end{equation}
Note that \eqref{eq:sf2} has been verified in \cite{KLT19}. To see
\eqref{eq:sf1}, set $z=\frac{a-b}{1-a\overline{b}}$ then 
\[
\left|1-a\overline{b}\right|=\frac{1-\left|b\right|^{2}}{\left|1+z\overline{b}\right|}\approx1-\left|b\right|^{2},
\]
since $\left|z\right|<\left.1\right/e$. From \cite[Proposition 6.1.1]{Kli91}, we
see that
\[
G_{\mathbb{D}\times\mathbb{D}}\left(F\left(w_{1},w_{2}\right),F\left(z_{1},z_{2}\right)\right)\leq G_{\mathbb{H}_{k}}\left(\left(w_{1},w_{2}\right),\left(z_{1},z_{2}\right)\right),
\]
for any $z=\left(z_{1},z_{2}\right),w=\left(w_{1},w_{2}\right)\in\mathbb{H}_{k}$,
where $F:\mathbb{H}_{k}\rightarrow\mathbb{D}\times\mathbb{D}$ is the
holomorphic map defined by
\[
F\left(z_{1},z_{2}\right)=\left(\frac{z_{1}^{k}}{z_{2}},z_{2}\right).
\]
Recall that 
\[
G_{\mathbb{D}\times\mathbb{D}}\left(w,z\right)=\max\left\{ \log\left|\frac{z_{1}-w_{1}}{1-w_{1}\overline{z}_{1}}\right|,\log\left|\frac{z_{2}-w_{2}}{1-w_{2}\overline{z}_{2}}\right|\right\} .
\]
It follows that on the set $A_{z}:=\left\{ w\in\mathbb{H}_{k}:G_{\mathbb{H}_{k}}\left(w,z\right)<-1\right\} $,
one has
\begin{equation}\label{eq:ques}
\left|\frac{\frac{z_{1}^{k}}{z_{2}}-\frac{w_{1}^{k}}{w_{2}}}{1-\frac{z_{1}^{k}\overline{w}_{1}^{k}}{z_{2}\overline{w}_{2}}}\right|<\frac{1}{e}\quad\text{and }\quad\left|\frac{z_{2}-w_{2}}{1-z_{2}\overline{w}_{2}}\right|<\frac{1}{e}.
\end{equation}
Therefore for $w\in A_{z}$,
\begin{align*}
\left|\mathcal{P}\left(z,w\right)\right|^{2} & =\frac{1}{\left|z_{2}w_{2}\right|^{2}\left|1-z_{2}\overline{w}_{2}\right|^{4}\left|1-\frac{z_{1}^{k}\overline{w}_{1}^{k}}{z_{2}\overline{w}_{2}}\right|^{4}}\\
 & \approx\frac{1}{\left|z_{2}w_{2}\right|^{2}\left(1-\left|w_{2}\right|^{2}\right)^{4}\left(1-\frac{\left|w_{1}\right|^{2k}}{\left|w_{2}\right|^{2}}\right)^{4}}\\
 & =\frac{\left|w_{2}\right|^{-2}\left(1-\left|w_{2}\right|^{2}\right)^{-2}\left(1-\frac{\left|w_{1}\right|^{2k}}{\left|w_{2}\right|^{2}}\right)^{-2}}{\left|z_{2}\right|^{2}\left(1-\left|w_{2}\right|^{2}\right)^{2}\left(1-\frac{\left|w_{1}\right|^{2k}}{\left|w_{2}\right|^{2}}\right)^{2}}\\
 & \approx\frac{\left|w_{2}\right|^{-2}\left(1-\left|w_{2}\right|^{2}\right)^{-2}\left(1-\frac{\left|w_{1}\right|^{2k}}{\left|w_{2}\right|^{2}}\right)^{-2}}{\left|z_{2}\right|^{2}\left(1-\left|z_{2}\right|^{2}\right)^{2}\left(1-\frac{\left|z_{1}\right|^{2k}}{\left|z_{2}\right|^{2}}\right)^{2}}\\
 & \approx K\left(z,z\right)K\left(w,w\right),
\end{align*}
here we have used the elementary fact \eqref{eq:sf1}-\eqref{eq:sf2}. 

For any $h\in A^{2}\left(\mathbb{H}_{k}\right)$, from the estimate
 \eqref{eq:He-Bo}, we
have
\[
\left|h\left(z\right)\right|^{2}\leq e^{4}\,K\left(z,z\right)\intop_{A_{z}}\left|h\left(w\right)\right|^{2}dV\left(w\right),\;\forall z\in\mathbb{H}_{k}.
\]
Thus
\begin{align*}
\intop_{\mathbb{H}_{k}}\left|h\left(z\right)\right|^{2}\,d\mu\left(z\right) & \lesssim\intop_{\mathbb{H}_{k}}\intop_{\mathbb{H}_{k}}\mathbbm{1}_{A_{z}}\left(w\right)K\left(z,z\right)\left|h\left(w\right)\right|^{2}dV\left(w\right)d\mu\left(z\right)\\
 & \lesssim\intop_{\mathbb{H}_{k}}\intop_{\mathbb{H}_{k}}\frac{\left|\mathcal{P}\left(z,w\right)\right|^{2}}{K\left(w,w\right)}\left|h\left(w\right)\right|^{2}d\mu\left(z\right)dV\left(w\right)\\
 & =\intop_{\mathbb{H}_{k}}T_{\mu}\left(w\right)\left|h\left(w\right)\right|^{2}dV\left(w\right)\\
 & \leq\left\Vert T_{\mu}\right\Vert _{L^{\infty}}\intop_{\mathbb{H}_{k}}\left|h\left(w\right)\right|^{2}dV\left(w\right).
\end{align*}
We conclude that $\mu$ is a Carleson measure on $\mathbb{H}_{k}$. 
\vspace*{0.3cm}
\item[$\quad\left(\text{iii}\right)\Rightarrow\left(\text{ii}\right)$. ] Since $\left|K\left(z,w\right)\right|\lesssim\left|\mathcal{P}\left(z,w\right)\right|$, this is immediate.
\vspace*{0.3cm}
\item[$\quad\left(\text{ii}\right)\Rightarrow\left(\text{i}\right)$. ] Let
\[
R\left(z,w\right):=q_{k}\left(s\right)+p_{k}\left(s\right)t+\frac{s^{k}}{t}p_{k}\left(s\right),
\]
where $s:=z_{1}\overline{w}_{1}$ and $t:=z_{2}\overline{w}_{2}$.
Then $R\left(\cdot,w\right)$ is holomorphic and $R\left(w,w\right)\geq1,\,\forall w\in \mathbb{H}_{k}$.
For any $h\in A^{2}\left(\mathbb{H}_{k}\right)$, we thus have
\begin{align*}
\left|h\left(z\right)\right|^{2} & \leq\left|h\left(z\right)\right|^{2}\left|R\left(z,z\right)\right|^{2}\\
 & \leq e^{4} K\left(z,z\right)\intop_{A_{z}}\left|h\left(w\right)\right|^{2}\left|R\left(w,z\right)\right|^{2}dV\left(w\right).
\end{align*}
So
\begin{align*}
\intop_{\mathbb{H}_{k}}\left|h\left(z\right)\right|^{2}d\mu\left(z\right) & \lesssim\intop_{\mathbb{H}_{k}}\intop_{\mathbb{H}_{k}}\mathbbm{1}_{A_{z}}\left(w\right)K\left(z,z\right)\left|R\left(w,z\right)\right|^{2}\left|h\left(w\right)\right|^{2}dV\left(w\right)d\mu\left(z\right)\\
 & \lesssim\intop_{\mathbb{H}_{k}}\intop_{\mathbb{H}_{k}}\mathbbm{1}_{A_{z}}\left(w\right)\frac{\left|\mathcal{P}\left(z,w\right)\right|^{2}}{K\left(w,w\right)}\left|R\left(w,z\right)\right|^{2}\left|h\left(w\right)\right|^{2}dV\left(w\right)d\mu\left(z\right)\\
 & \lesssim\intop_{\mathbb{H}_{k}}\intop_{\mathbb{H}_{k}}\frac{\left|K\left(z,w\right)\right|^{2}}{K\left(w,w\right)}\left|h\left(w\right)\right|^{2}dV\left(w\right)d\mu\left(z\right)\\
 & \lesssim\left\Vert \mathcal{B}_{\mu}\right\Vert _{L^{\infty}}\intop_{\mathbb{H}_{k}}\left|h\left(w\right)\right|^{2}dV\left(w\right),
\end{align*}
as desired.
\hfill  \hfill $\square$
\end{enumerate}

\section{Proof of Theorem \ref{thm:2}} \label{prfB}

\begin{enumerate} [leftmargin=0pt, align=left, labelwidth=!, labelsep=0pt]
\item[$\quad\left(\text{i}\right)\Rightarrow\left(\text{iii}\right)$. ]  It suffices to show that
\[
\frac{\mathcal{P}\left(\cdot,w\right)\left|w_{2}\right|}{\sqrt{K\left(w,w\right)}}\rightarrow0\;\text{weakly in }A^{2}\left(\mathbb{H}_{k}\right)\;\text{as }w\rightarrow\partial\mathbb{H}_{k}.
\]
Take any $g\in A^{2}\left(\mathbb{H}_{k}\right)$, and choose a sequence
of domains $\left\{ \Omega_{j}\right\} $ such that $\overline{\Omega}_{j}\Subset\Omega_{j+1}\Subset\mathbb{H}_{k}$
and $\cup_{j=1}^{\infty}\Omega_{j}=\mathbb{H}_{k}$. For each $j$,
since
\[
\left|\mathcal{P}\left(z,w\right)\right|=\frac{1}{\left|z_{2}w_{2}\right|\left|1-z_{2}\overline{w}_{2}\right|^{2}\left|1-\frac{z_{1}^{k}\overline{w}_{1}^{k}}{z_{2}\overline{w}_{2}}\right|^{2}},
\]
there exists $c_{j}>0$ such that
\[
\left|\mathcal{P}\left(z,w\right)\right|<\frac{c_{j}}{\left|w_{2}\right|},\;\forall z\in\Omega_{j},\;\forall w\in\mathbb{H}_{k}.
\]
Thus
\begin{align}
\left|\,\intop_{\mathbb{H}_{k}}\frac{\mathcal{P}\left(z,w\right)\left|w_{2}\right|}{\sqrt{K\left(w,w\right)}}\,\overline{g\left(z\right)}\,dV\left(z\right)\right| & \leq\frac{c_{j}\sqrt{\left|\mathbb{H}_{k}\right|}}{\sqrt{K\left(w,w\right)}}\left\Vert g\right\Vert _{L^{2}\left(\mathbb{H}_{k}\right)}+\left\Vert \frac{\mathcal{P}\left(\cdot,w\right)}{\sqrt{K\left(w,w\right)}}\right\Vert _{L^{2}\left(\mathbb{H}_{k}\right)}\hspace*{-0.25cm}\left\Vert g\right\Vert _{L^{2}\left(\left.\mathbb{H}_{k}\right\backslash \Omega_{j}\right)}\nonumber \\
 & \leq\frac{c_{j}\sqrt{\left|\mathbb{H}_{k}\right|}}{\sqrt{K\left(w,w\right)}}\left\Vert g\right\Vert _{L^{2}\left(\mathbb{H}_{k}\right)}+C\left\Vert g\right\Vert _{L^{2}\left(\left.\mathbb{H}_{k}\right\backslash \Omega_{j}\right)}.\label{eq:co1}
\end{align}
Note that the right hand side of \eqref{eq:co1} can be made arbitrarily
small as $w\rightarrow\partial\mathbb{H}_{k}$ and $j\rightarrow\infty$ 
because $g\in L^{2}\left(\mathbb{H}_{k}\right)$ and 
\[
\lim_{w\rightarrow\partial\mathbb{H}_{k}}\frac{1}{\sqrt{K\left(w,w\right)}}=0.
\]
It follows that
\[
\intop_{\mathbb{H}_{k}}\frac{\mathcal{P}\left(z,w\right)\left|w_{2}\right|}{\sqrt{K\left(w,w\right)}}\,\overline{g\left(z\right)}\,dV\left(z\right)\rightarrow0\;\text{as }w\rightarrow\partial\mathbb{H}_{k},
\]
as desired.\vspace*{0.3cm}
\item[$\quad\left(\text{i}\right)\Rightarrow\left(\text{ii}\right)$. ] 
Since $\left|K\left(z,w\right)\right|\lesssim\left|\mathcal{P}\left(z,w\right)\right|$, this is straightforward from the argument in $\quad\left(\text{i}\right)\Rightarrow\left(\text{iii}\right)$.
\vspace*{0.3cm}
\item[$\quad\left(\text{ii}\right)\Rightarrow\left(\text{i}\right)$. ] 
Following the proof of Theorem \ref{thm:1}, we first establish the following estimate:
\begin{lemma}\label{lem:1}
For any $\varepsilon>0$, there exists $\delta_{\varepsilon}>0$ such
that
\[
\intop_{W_{\delta_{\varepsilon}}}\left|h\left(z\right)\right|^{2}dV\left(z\right)<\varepsilon\intop_{\mathbb{H}_{k}}\left|h\left(z\right)\right|^{2}dV\left(z\right),\;\forall h\in A^{2}\left(\mathbb{H}_{k}\right),
\]
where $W_{\delta_{\varepsilon}}:=\left\{ z\in\mathbb{H}_{k}:\left|z_{2}\right|<\delta_{\varepsilon}\right\} $.

\end{lemma}

\begin{proof}[Proof of Lemma \ref{lem:1}]\;
Repeating the argument used in the proof of Theorem \ref{thm:1}, we obtain that
\begin{align*}
\intop_{W_{\delta_{\varepsilon}}}\left|h\left(z\right)\right|^{2}dV\left(z\right) & \lesssim\intop_{\mathbb{H}_{k}}\intop_{\mathbb{H}_{k}}\mathbbm{1}_{W_{\delta_{\varepsilon}}}\left(z\right)\mathbbm{1}_{A_{z}}\left(w\right) K\left(z,z\right)\left|h\left(w\right)\right|^{2}dV\left(w\right)dV\left(z\right)\\
 & \lesssim\intop_{\mathbb{H}_{k}}\left(\;\intop_{\mathbb{H}_{k}}\mathbbm{1}_{W_{\delta_{\varepsilon}}}\left(z\right)\frac{\left|\mathcal{P}\left(z,w\right)\right|^{2}}{K\left(w,w\right)}\,dV\left(z\right)\right)\left|h\left(w\right)\right|^{2}dV\left(w\right).
\end{align*}
It remains to verify the existence of $\delta_{\varepsilon}$ such
that
\[
\intop_{\mathbb{H}_{k}}\mathbbm{1}_{W_{\delta_{\varepsilon}}}\left(z\right)\frac{\left|\mathcal{P}\left(z,w\right)\right|^{2}}{K\left(w,w\right)}\,dV\left(z\right)\lesssim\varepsilon,\;\forall w\in\mathbb{H}_{k}.
\]
We have
\begin{align*}
\intop_{W_{\delta_{\varepsilon}}}\left|\mathcal{P}\left(z,w\right)\right|^{2}\,dV\left(z\right) & \leq\frac{\left|w_{2}\right|^{-2}}{k}\left(2^{4}\intop_{\left|z_{2}\right|<\delta_{\varepsilon}}\left|z_{2}\right|^{\frac{2}{k}-2}dV\left(z_{2}\right)\right)\left(\intop_{\mathbb{D}}\frac{\left|\xi\right|^{\frac{2}{k}-2}}{\left|1-\xi\frac{\overline{w}_{1}^{k}}{\overline{w}_{2}}\right|^{4}}dV\left(\xi\right)\right)\\
 & \lesssim\delta_{\varepsilon}^{\frac{2}{k}}\left|w_{2}\right|^{-2}\left(1-\left|w_{2}\right|^{2}\right)^{-2}\left(1-\frac{\left|w_{1}\right|^{2k}}{\left|w_{2}\right|^{2}}\right)^{-2}\\
 & \lesssim\varepsilon\,K\left(w,w\right),
\end{align*}
provided that $\delta_{\varepsilon}<\min\left\{ \varepsilon^{k/2},\left.1\right/2\right\} $, as
desired.

\end{proof}

We are ready to verify the implication $\left(\text{ii}\right)\Rightarrow\left(\text{i}\right)$. Choose a sequence of domains $\left\{ \Omega_{j}\right\} $ such that
$\overline{\Omega}_{j}\Subset\Omega_{j+1}\Subset\mathbb{H}_{k}$ and
$\cup_{j=1}^{\infty}\Omega_{j}=\mathbb{H}_{k}$. For each $j$, let
\[
R_{j}\left(w\right):=\intop_{\mathbb{H}_{k}}\frac{\left|K\left(z,w\right)\right|^{2}}{K\left(w,w\right)}\left(1-\mathbbm{1}_{\Omega_{j}}\left(z\right)\right)\,d\mu\left(z\right),\;w\in\mathbb{H}_{k}.
\]
Note that  $R_{j}$ and $\mathcal{B}_{\mu}$ are continuous on $\mathbb{H}_{k}$. Consider the operators 
\begin{align*}
i_{\Omega_{j}}:A^{2}\left(\mathbb{H}_{k}\right) & \rightarrow L^{2}\left(\mathbb{H}_{k},\mu\right)\\
h & \rightarrow\mathbbm{1}_{\Omega_{j}}h.
\end{align*}
By Montel's theorem, $i_{\Omega_{j}}$ are compact operators. It remains to show that  \begin{equation}
\left\Vert i_{\Omega_{j}}-i_{\mathbb{H}_{k}}\right\Vert \xrightarrow{j\rightarrow\infty}0.\label{eq:nomcon}
\end{equation}
Fix $\varepsilon>0$, and choose $\delta_{\varepsilon}>0$ as in Lemma \ref{lem:1}. Then for any $h\in A^{2}\left(\mathbb{H}_{k}\right)$ such that $\left\Vert h\right\Vert _{L^{2}\left(\mathbb{H}_{k}\right)}=1$, 
we have 
\[
\intop_{W_{\delta_{\varepsilon}}}\left|h\left(w\right)\right|^{2}dV\left(w\right)<\varepsilon.
\]
The key point we need here is that $\delta_{\varepsilon}$ is independent of $h$. Note that $R_{j}\leq \mathcal{B}_{\mu}$ on $\mathbb{H}_{k}$. And, on $\left.\mathbb{H}_{k}\right\backslash W_{\delta_{\varepsilon}}$,
we have $\left|z_{2}\right|\geq\delta_{\varepsilon}$.
The hypothesis implies that $R_{j}$ can be extended continuously to the compact set $\left.\overline{\mathbb{H}}_{k}\right\backslash W_{\delta_{\varepsilon}}$ (with zero value on the part $\left.\partial\mathbb{H}_{k}\right\backslash W_{\delta_{\varepsilon}}$). Moreover, $R_{j}\searrow0$ on $\left.\overline{\mathbb{H}}_{k}\right\backslash W_{\delta_{\varepsilon}}$
by the dominated convergence theorem. From Dini's
theorem, there exists $j_{0}>0$ such that $\left|R_{j}\left(w\right)\right|<\varepsilon$,
for any $w\in\left.\overline{\mathbb{H}}_{k}\right\backslash W_{\delta_{\varepsilon}}$
and $j>j_{0}$. Following the last lines in the proof of Theorem \eqref{thm:1},
we see that
\begin{align*}
\left\Vert \left(1-\mathbbm{1}_{\Omega_{j}}\right)h\right\Vert _{L^{2}\left(\mathbb{H}_{k},\mu\right)}^{2} & \leq\intop_{\mathbb{H}_{k}}\intop_{\mathbb{H}_{k}}\frac{\left|K\left(z,w\right)\right|^{2}}{K\left(w,w\right)}\left(1-\mathbbm{1}_{\Omega_{j}}\left(z\right)\right)\left|h\left(w\right)\right|^{2}d\mu\left(z\right)dV\left(w\right)\\
 & \leq\intop_{W_{\delta_{\varepsilon}}}\mathcal{B}_{\mu}\left(w\right)\left|h\left(w\right)\right|^{2}dV\left(w\right)+\intop_{\left.\mathbb{H}_{k}\right\backslash W_{\delta_{\varepsilon}}}R_{j}\left(w\right)\left|h\left(w\right)\right|^{2}dV\left(w\right)\\
 & \leq\left\Vert \mathcal{B}_{\mu}\right\Vert _{L^{\infty}\left(\mathbb{H}_{k}\right)}\,\varepsilon+\varepsilon,
\end{align*}
for $j>j_{0}$. This demonstrates the claim \eqref{eq:nomcon}.
\vspace*{0.3cm}

\item[$\quad\left(\text{iii}\right)\Rightarrow\left(\text{ii}\right)$. ] Since  $\mathcal{B}_{\mu}\leq T_{\mu}$, this is immediate.
\hfill  \hfill $\square$

\vspace*{0.3cm}
\end{enumerate}

\begin{remark}\label{remar} \normalfont
With the characterisations in Theorem \ref{thm:1} and Theorem \ref{thm:2}, we now discuss some specific examples. Consider a special measure $\mu$ on $\mathbb{H}_{k}$ in the form
\[
d\mu\left(z\right)=f\left(\frac{z_{1}^{k}}{z_{2}}\right)g\left(z_{2}\right)dV\left(z\right),
\]
where $f$ and $g$ are some non-negative, measurable functions on $\mathbb{D}$. Then
we get that
\[
T_{\mu}\left(w\right)\approx\frac{1}{K\left(w,w\right)}\frac{1}{\left|w_{2}\right|^{2}}\left(\intop_{\mathbb{D}}\frac{g\left(z_{2}\right)\left|z_{2}\right|^{\frac{2}{k}-2}}{\left|1-z_{2}\overline{w}_{2}\right|^{4}}dV\left(z_{2}\right)\right)\left(\intop_{\mathbb{D}}\frac{f\left(\xi\right)\left|\xi\right|^{\frac{2}{k}-2}}{\left|1-\xi\frac{\overline{w}_{1}^{k}}{\overline{w}_{2}}\right|^{4}}dV\left(\xi\right)\right).
\]
Note that 
\[
\frac{1}{K\left(w,w\right)\left|w_{2}\right|^{2}}\approx\left(1-\left|w_{2}\right|^{2}\right)^{2}\left(1-\frac{\left|w_{1}\right|^{2k}}{\left|w_{2}\right|^{2}}\right)^{2}.
\]
It follows that $\mu$ is a (non-zero) Carleson measure on $\mathbb{H}_{k}$
if and only if both measures
\[
d\mu_{1}\left(\omega\right):=f\left(\omega\right)\left|\omega\right|^{\frac{2}{k}-2}dV\left(\omega\right)\quad\text{and }\;\;d\mu_{2}\left(\omega\right):=g\left(\omega\right)\left|\omega\right|^{\frac{2}{k}-2}dV\left(\omega\right)
\]
are Carleson measures on $\mathbb{D}$. 

If we set
\[
f\left(\xi\right):=\left(1-\left|\xi\right|^{2}\right)^{3}\left|\xi\right|^{2-\frac{2}{k}}\quad\text{and}\quad g\left(z_{2}\right):=\left(1-\left|z_{2}\right|^{2}\right)^{3}\left|z_{2}\right|^{2-\frac{2}{k}}
\]
then Forelli-Rudin estimates (\cite{LiuC15}) give that 
\[
T_{\mu}\left(w\right)\approx\frac{1}{K\left(w,w\right)}\frac{1}{\left|w_{2}\right|^{2}}.
\]
Thus $T_{\mu}\in L^{\infty}\left(\mathbb{H}_{k}\right)$ and $\left|w_{2}\right|^{2}T_{\mu}\left(w\right)\rightarrow0$
as $w\rightarrow\partial\mathbb{H}_{k}$. So the measure
\[
d\mu\left(z\right)=\left|z_{1}\right|^{2k-2}\left(1-\left|z_{2}\right|^{2}\right)^{3}\left(1-\frac{\left|z_{1}\right|^{2k}}{\left|z_{2}\right|^{2}}\right)^{3}dV\left(z_{1},z_{2}\right)
\]
is a vanishing Carleson measure on $\mathbb{H}_{k}$.

If we consider
\[
f\left(\xi\right):=\left|\xi\right|^{2-\frac{2}{k}}\quad\text{and}\quad g\left(z_{2}\right):=\left|z_{2}\right|^{2-\frac{2}{k}}
\]
then
\begin{align*}
T_{\mu}\left(w\right) & \approx\frac{1}{K\left(w,w\right)}\frac{1}{\left|w_{2}\right|^{2}}\left(1-\left|w_{2}\right|^{2}\right)^{-2}\left(1-\frac{\left|w_{1}\right|^{2k}}{\left|w_{2}\right|^{2}}\right)^{-2}\\
 & \approx1.
\end{align*}
On the other hand, consider the point $M\left(\left.1\right/2,1\right)\in\partial\mathbb{H}_{k}$.
For any $\delta>0$, since $\left|w_{2}\right|^{\delta}T_{\mu}\left(w\right)\approx\left|w_{2}\right|^{\delta}$,
it follows that $\left|w_{2}\right|^{\delta}T_{\mu}\left(w\right)\not\rightarrow0$
as $w\rightarrow M$.  So
\[
d\mu\left(z\right)=\left|z_{1}\right|^{2k-2}dV\left(z_{1},z_{2}\right)
\]
 is a Carleson measure but not a vanishing Carleson measure on $\mathbb{H}_{k}$.

\end{remark}

\section{Concluding remarks}

For the classical Hartogs triangle $\mathbb{H}$, since $\left|K\left(z,w\right)\right|\approx\left|\mathcal{P}\left(z,w\right)\right|$,
it can be seen from the proof of Theorem \ref{thm:1} that  the Skwarczy{\'n}ski distance
$d_{S}\left(z,w\right)$ is bounded by some constant $c<1$ for any
$w\in A_{z,1}$. However, it is not clear to us whether this is true
for $\mathbb{H}_{k}$, with $k\geq2$. Note that if we rely only on
the weaker estimates \eqref{eq:ques} then this is not the case. For example,
consider $k=2$ and $z_{j}=\left(-\left(1-\frac{1}{j}\right),1-\frac{1}{j}\right),$
$w_{j}=\left(1-\frac{1}{j},1-\frac{1}{j}\right)$, for $j\geq1$, 
then the estimates \eqref{eq:ques} hold, however $d_{S}\left(z_{j},w_{j}\right)\rightarrow1$
as $j\rightarrow\infty$.

We do not know an example of a bounded pseudoconvex domain such that
the condition $\mathcal{B}_{\mu}\in L^{\infty}$ is not sufficient
for $\mu$ being a Carleson measure. Here we make the trivial observation that for any bounded smooth domain (not necessarily pseudoconvex) $\Omega$
in $\mathbb{C}^{n}$, the condition $\mathcal{B}_{\mu}\in L^{\infty}\left(\Omega\right)$
implies certain $L^{2}$ regularity. Indeed, for any 
Borel measure $\mu$,
the inclusion $W^{s}\left(\Omega\right)\cap\mathcal{O}\left(\Omega\right)\hookrightarrow L^{2}\left(\Omega,\mu\right)$
is bounded if $s\geq\left.\left(3n+1\right)\right/2$. Here $W^{s}\left(\Omega\right)$ is the standard $L^{2}$ Sobolev space of order $s$ on $\Omega$.
To check this, from \cite[Lemma 2]{Bel81-1} and the remark after it, there exists a bounded linear operator $\Phi^{s}:W^{s}\left(\Omega\right)\cap\mathcal{O}\left(\Omega\right)\rightarrow W_{0}^{s}\left(\Omega\right)$ such that $P\Phi^{s}=I$, where $P$ denotes the Bergman projection of $\Omega$. For $h\in W^{s}\left(\Omega\right)\cap\mathcal{O}\left(\Omega\right)$, it thus follows that 
\begin{align*}
\intop_{\Omega}\left|h\left(z\right)\right|^{2}d\mu\left(z\right) & =\intop_{\Omega}\left|\intop_{\Omega}K\left(z,w\right)\Phi^{s}\left(h\right)\left(w\right)dV\left(w\right)\right|^{2}d\mu\left(z\right)\\
 & \leq\left(\intop_{\Omega}\left(\,\intop_{\Omega}\left|K\left(z,w\right)\right|^{2}\left|\Phi^{s}\left(h\right)\left(w\right)\right|^{2}d\mu\left(z\right)\right)^{\frac{1}{2}}dV\left(w\right)\right)^{2}\\
 & \leq\left\Vert \mathcal{B}_{\mu}\right\Vert _{L^{\infty}}\left(\,\intop_{\Omega}K^{\frac{1}{2}}\left(w,w\right)\left|\Phi^{s}\left(h\right)\left(w\right)\right|dV\left(w\right)\right)^{2}.
\end{align*}
Since $\Phi^{s}\left(h\right)\in W_{0}^{s}\left(\Omega\right)$, by Sobolev embedding theorem we have  
\[ \left|\Phi^{s}\left(h\right)\left(w\right)\right|\leq C_{\Omega}\left\Vert \Phi^{s}\left(h\right)\right\Vert _{W^{\left.\left(3n+1\right)\right/2}\left(\Omega\right)}\delta_{\Omega}^{\frac{n+1}{2}}\left(w\right),\;\forall w\in\Omega. \] 
Here $\delta_{\Omega}$ denotes the distance function to the boundary. On the other hand, 
\[ K^{\frac{1}{2}}\left(w,w\right)\leq C_{\Omega}\,\delta_{\Omega}^{-\frac{n+1}{2}}\left(w\right),\;\forall w\in\Omega. \] 
We arrive at \[ \intop_{\Omega}\left|h\left(z\right)\right|^{2}d\mu\left(z\right)\leq const.\left\Vert \Phi^{s}\left(h\right)\right\Vert _{W^{\left.\left(3n+1\right)\right/2}\left(\Omega\right)}^{2}\leq const.\left\Vert h\right\Vert _{W^{s}\left(\Omega\right)}^{2} \] as desired.


\textbf{Acknowledgement.} This work is supported by Ho Chi Minh City University of Technology.


\end{document}